\documentclass[10pt]{amsart}
\usepackage[utf8]{inputenc}
\usepackage{amsmath}
\usepackage{amsthm}
\usepackage{amssymb}
\usepackage{mathrsfs}
\usepackage[dvipsnames]{xcolor}
\usepackage{verbatim}
\usepackage{soul}
\usepackage[shortlabels]{enumitem}
\usepackage{mathtools}
\usepackage{tikz}
\usetikzlibrary{decorations.markings, arrows,shapes,positioning}

\tikzstyle{vertex}=[circle, draw, inner sep=0pt, minimum size=6pt]
\usepackage{hyperref}

\usepackage[margin=1in]{geometry}

\usepackage{graphicx}

\newcommand{\cF}{\mathcal{F}}

\newcommand{\cB}{\mathcal{B}}

\newcommand{\Z}{\mathbb{Z}}

\newcommand{\Aut}{\operatorname{Aut}}

\newcommand{\AugBerg}[1]{\Delta_{#1}}
\newcommand{\Berg}[1]{\underline{\Delta}_{#1}}
\newcommand{\indepsets}[1]{\mathcal{I}(#1)}
\newcommand{\bases}[1]{\mathcal{B}(#1)}
\newcommand{\flats}[1]{\mathcal{F}(#1)}

\newcommand{\Cone}{\operatorname{Cone}}
\newcommand{\sgn}{\operatorname{sgn}}

\newcommand{\subdot}{\bullet}

\theoremstyle{plain}

\theoremstyle{plain}
\newtheorem{thm}{Theorem}[section]
\newtheorem{theorem}{Theorem}[section]

\newtheorem{lemma}[theorem]{Lemma}
\newtheorem{cor}[theorem]{Corollary}

\theoremstyle{definition}
\newtheorem{definition}[theorem]{Definition}

\newtheorem{example}[theorem]{Example}

\newtheorem{remark}[theorem]{Remark}

\title{Topology of augmented Bergman complexes}
\author[]{Elisabeth Bullock, Aidan Kelley, Victor Reiner, Kevin Ren,\\ Gahl Shemy, Dawei Shen, Brian Sun, Amy Tao, Zhichun Joy Zhang}

\address{MIT}
\email{edb22@mit.edu}

\address{Washington University in St. Louis}
\email{aidankelley@wustl.edu}

\address{University of Minnesota- Twin Cities}
\email{reiner@umn.edu}

\address{MIT}
\email{kevinren@mit.edu}

\address{Univ. of California- Santa Barbara}
\email{gahlshemy@ucsb.edu}

\address{Washington University in St. Louis}
\email{Shen.dawei@wustl.edu}

\address{Yale University}
\email{brian.sun@yale.edu}

\address{Wellesley College}
\email{atao3@wellesley.edu}

\address{Swarthmore College}
\email{zzhang3@swarthmore.edu}



\date{August 30, 2021}

\begin{document}

\maketitle

\begin{abstract}
  The augmented Bergman complex of a matroid is a simplicial complex introduced recently in work of Braden, Huh, Matherne, Proudfoot and Wang. It may be viewed as a hybrid of two well-studied pure shellable simplicial complexes associated to matroids: the independent set complex and Bergman complex. 
  
  It is shown here that the augmented Bergman complex is also shellable, via two different families of 
  shelling orders. Furthermore, comparing the description of its homotopy type induced from the two shellings re-interprets a known convolution formula counting bases of the matroid. The representation of the automorphism group of the matroid on the homology of the augmented Bergman complex turns out to have a surprisingly simple description. This last fact is generalized to closures beyond those coming from a matroid.
\end{abstract}


\section{Introduction}
\label{intro-section}

{\it Matroids} are an abstraction of the combinatorial properties of linear dependence for a set of vectors in a vector space, introduced independently in the 1930s by H. Whitney and T. Nakasawa. Here we will work exclusively with a matroid $M$ on a {\it finite} ground set $E$,
which can be specified by either of these two
subcollections of the subsets $2^E$:
\begin{itemize}
    \item the {\it independent sets} $\indepsets{M}$,
    modeling the subsets of the vectors that are linearly independent, or 
    \item the {\it flats} $\flats{M}$, modeling the subsets of the vectors that are closed under taking linear span.
\end{itemize}
The independent sets and flats satisfy certain axioms {\sf I1, I2, I3} and {\sf F1, F2, F3},
recalled in Section~\ref{background-section} below;
see the book by Oxley \cite{Oxley} for further background on matroids. Both $\indepsets{M}$ and $\flats{M}$ give rise
to well-studied abstract simplicial complexes associated to $M$, as we now explain. 

The first two axioms {\sf I1, I2} for independent sets imply that $\indepsets{M}$ forms an {\it abstract simplicial complex} on vertex set $E$. The third axiom {\sf I3} then implies that all inclusion-maximal
independent sets (called {\it bases}) have the same
cardinality $r(M)$, called the {\it rank} of the matroid. Thus the simplicial complex $\indepsets{M}$ is {\it pure},
meaning that all of its inclusion-maximal faces, called {\it facets}, have the same dimension $r(M)-1$.

On the other hand, one usually considers the collection of flats $\flats{M}$ as a  partial order via inclusion, and then one can construct its {\it order complex} $\Delta \flats{M}$ as the simplicial
complex whose vertices are the flats, and
whose simplices are the linearly ordered collections of flats. There are two extreme flats $\overline{\varnothing}, E$ which are
comparable to all the other flats, so that they form
cone vertices in this simplicial complex $\Delta \flats{M}$. One usually removes these two cone vertices to obtain the topologically more
interesting
{\it Bergman complex}\footnote{Also sometimes known as the order complex of the {\it proper part} $\flats{M} \setminus \{\overline{\varnothing},E\}$; see Subsection~\ref{order-complex-subsection}.} $\Berg{M}:=\Delta( \flats{M} \setminus \{\overline{\varnothing},E\} )$, which turns out to
be pure of dimension $r(M)-2$.

Both $\indepsets{M}$ and $\Berg{M}$
were proven around 1980 (in work of Provan and Billera \cite{BilleraProvan} and of Garsia \cite{Garsia}) to be not only pure, but also {\it shellable}, a property with strong topological and algebraic consequences. In particular,
a pure $d$-dimensional shellable complex $\Delta$ is homotopy equivalent to a $\beta$-fold wedge of $d$-dimensional
spheres, where $\beta$ is its top (reduced homology) {\it Betti number},
or the absolute value $|\tilde{\chi}(\Delta)|$ of its (reduced) {\it Euler characteristic}. Shellability can be viewed as a connectivity property much stronger than {\it gallery-connectedness}.
The hierarchy of purity, gallery-connectedness and shellability are reviewed in Subsection~\ref{pure-gallery-connected-shellable-section} below.

Our goal here is to study the topology of the 
following ``hybrid'' of $\indepsets{M}$ and $\Berg{M}$, introduced in the monumental recent work of
Braden, Huh, Matherne, Proudfoot and Wang \cite{BHMPW1,BHMPW2} that resolved several important conjectures in matroid theory.

\vskip.1in
\noindent
{\bf Definition.}
Given a matroid $M$ on ground set $E$,
the {\it augmented Bergman complex} $\AugBerg{M}$ is
the abstract simplicial complex on
vertex set 
$$
\{y_i: i \in E \} \sqcup \{x_F: \text{proper flats }F \in \flats{M} \setminus \{E\} \}
$$
whose simplices are the subsets
\begin{equation}
\label{typical-augberg-simplex}
\{y_i\}_{i \in I} 
\cup
\{x_{F_1},x_{F_2},\ldots,x_{F_\ell} \}
\end{equation}
for which $I \in \indepsets{M}$
and the proper flats $F_i$ satisfy
$
I \subseteq 
 F_1 \subsetneq 
  F_2 \subsetneq \cdots
  \subsetneq F_\ell \;\; (\subsetneq E).
$

It is noted in \cite[\S2 Prop. 2.3]{BHMPW1} that  $\AugBerg{M}$ is pure of dimension $r(M)-1$, that it is
gallery-connected, and that it 
contains as full-dimensional subcomplexes both the
independent set complex $\indepsets{M}$ (as the simplices in
\eqref{typical-augberg-simplex} with $\ell=0$), and 
the cone over the
Bergman complex $\Berg{M}$ with cone vertex $x_{\overline{\varnothing}}$
(as the simplices in \eqref{typical-augberg-simplex} with $\#I = 0$);  we will denote the latter complex by $\Cone(\Berg{M})=\Delta(\cF\setminus \{E\})$.

This motivates our first main result, proven in Section~\ref{shellability-section}.

\begin{theorem}
\label{shelling-theorem}
The augmented Bergman complex $\AugBerg{M}$ of a matroid is shellable,
via two families of shellings:
\begin{itemize}
    \item [(i)] one family that shells the facets of $\Cone(\Berg{M})$
    first, and the facets of $\indepsets{M}$
     last, 
       \item [(ii)] one family that shells the facets of $\indepsets{M}$ first, and the facets of $\Cone(\Berg{M})$  last.
\end{itemize}
\end{theorem}

Shellability immediately implies that the geometric realization of $\AugBerg{M}$ is homotopy equivalent to a wedge of $(r(M)-1)$-dimensional spheres, and gives
a combinatorial expression for the number of
spheres. For example, the  
Provan-Billera shelling of $\indepsets{M}$ shows
that it is homotopy equivalent to a wedge of $(r(M)-1)$-spheres, and the number of spheres in
the wedge is the evaluation $T_M(0,1)$ of
the famous {\it Tutte polynomial} $T_M(x,y)$; see Bj\"orner \cite[\S 7.3]{Bjorner-chapter}.
Similarly, Garsia's shelling of
$\Berg{M}$ shows that it is homotopy equivalent to a $T_M(1,0)$-fold 
wedge of $(r(M)-2)$-dimensional spheres; 
see Bj\"orner \cite[\S 7.4, 7.6]{Bjorner-chapter}.

Theorem~\ref{shelling-theorem} gives the following description of the homotopy type
of $\AugBerg{M}$, proven in Section~\ref{homotopy-type-section}.  It
involves the set $\bases{M}$ of {\it bases} of $M$,
which we recall are the maximal independent sets, indexing the facets of $\indepsets{M}$. 

\begin{cor}
\label{homotopy-type-corollary}
For a matroid $M$, the augmented Bergman complex $\AugBerg{M}$
is homotopy equivalent to a $\beta$-fold wedge of
$(r(M)-1)$-spheres, with two different expressions for $\beta$:
\begin{align}
\label{flag-to-basis-expression-for-beta}
\beta&=\#\bases{M}\quad (=T_M(1,1))\\
\label{basis-to-flag-expression-for-beta}
\beta&=\sum_{F \in \flats{M}} T_{M|_F}(0,1) \cdot T_{M/F}(1,0).
\end{align}
Expressions \eqref{flag-to-basis-expression-for-beta},
\eqref{basis-to-flag-expression-for-beta} for $\beta$ 
are predicted by the shellings in 
Theorem~\ref{shelling-theorem}(i),(ii),
respectively.
\end{cor}

\noindent
Here $M|_F$ and $M/F$ are the matroids obtained from $M$
by {\it restriction} to the flat $F$ and {\it contraction}
on the flat $F$, respectively; see Subsection~\ref{matroid-subsection} below.
We remark that the concordance between
the two expressions for $\beta$ in
Corollary~\ref{homotopy-type-corollary} comes from
specializing a well-known Tutte polynomial convolution formula \cite{EtienneLasVergnas, KookReinerStanton}:
$$
T_M(x,y)=\sum_{F \in \flats{M}} T_{M|_F}(0,y) \cdot T_{M/F}(x,0).
$$

Our last main result is a surprisingly simple description for the representation
on the homology of the augemented Bergman
complex given by the action of the matroid's {\it automorphism group} 
$$
\begin{aligned}
\Aut(M)&:=
\{\text{bijections }E \overset{\sigma}{\longrightarrow} E: \sigma(I) \in \indepsets{M} \text{ for all }I \in \indepsets{M}\}\\
&=\{\text{bijections }E \overset{\sigma}{\longrightarrow} E: \sigma(F) \in \flats{M} \text{ for all }I \in \flats{M}\}.
\end{aligned}
$$

\begin{cor}
\label{homology-rep-theorem}
The action of $\Aut(M)$ on the top (reduced)
homology $\tilde{H}_{r(M)-1}(\AugBerg{M},\Z)$
is the same as its action on the top
oriented simplicial chain group $\tilde{C}_{r(M)-1}(\indepsets{M},\Z)$ for the
complex $\indepsets{M}$.
\end{cor}

\noindent
More explicitly, this means that the
action is a signed permutation representation of $\Aut(M)$, in which $\Z$-basis elements $[B]$ indexed by bases $B$ in $\bases{M}$ are permuted up to (explicit) signs; see Remark~\ref{signed-permutation-rep-remark} below.

Corollary~\ref{homology-rep-theorem} is generalized
in Section~\ref{other-closures-section} to a statement
(Theorem~\ref{closure-augberg-homology-representation-theorem}) 
that applies beyond matroids, to a finite set $E$ equipped with an
arbitrary {\it closure operator} 
$f: 2^E \longrightarrow 2^E$.

\begin{example}
\label{uniform-matroid-example}
Let $M$ be the uniform matroid of rank $2$
on ground set $E=\{1,2,3\}$, which has
 three bases $\cB=\{\{1,2\},\{1,3\},\{2,3\}\}$.
Figure \ref{exampleofM1} depicts the Hasse diagram of its lattice of flats $\cF$ on the left, along
with the simplicial complexes $\indepsets{M}, \Cone(\Berg{M})$ and $\AugBerg{M}$ from left to right.
All three are all (pure) $1$-dimensional complexes, that is, graphs.
For graphs, shellability is equivalent to
being connected.

\begin{figure}[h]
\begin{center}
\begin{minipage}{.15\textwidth}
\begin{tikzpicture}
  [scale=.33,auto=left,every node/.style={circle,fill=black!20}]
  \node (n123) at (3,8) {123};
  \node (n1) at (0,4)  {1};
  \node (n2) at (3,4)  {2};
  \node (n3) at (6,4)  {3};
  \node (n0) at (3,0) {$\varnothing$};
   \foreach \from/\to in {n123/n1, n123/n2, n123/n3, n1/n0, n2/n0, n3/n0}
    \draw (\from) -- (\to);
\end{tikzpicture}
\end{minipage}
\quad \quad \quad
\begin{minipage}{.2\textwidth}
\begin{tikzpicture}
  [scale=.33,auto=left,every node/.style={circle,fill=black!20}]
  \node (y1) at (4,8) {$y_1$};
  \node (y2) at (0,0) {$y_2$};
  \node (y3) at (8,0) {$y_3$};
  \foreach \from/\to in {y1/y2, y1/y3, y2/y3}
    \draw (\from) -- (\to);
\end{tikzpicture}
\end{minipage}
\quad
\begin{minipage}{.2\textwidth}
\begin{tikzpicture}
  [scale=.33,auto=left,every node/.style={circle,fill=black!20}]
  \node (x0) at (6,6) {$x_{\varnothing}$};
  \node (x1) at (6,9)  {$x_1$};
  \node (x2) at (3,3)  {$x_2$};
  \node (x3) at (9,3) {$x_3$};
  \foreach \from/\to in {x0/x1, x0/x2,x0/x3}
    \draw (\from) -- (\to);
\end{tikzpicture}
\end{minipage}
\begin{minipage}{.2\textwidth}
\begin{tikzpicture}
  [scale=.33,auto=left,every node/.style={circle,fill=black!20}]
  \node (x0) at (6,6) {$x_{\varnothing}$};
  \node (x1) at (6,9)  {$x_1$};
  \node (x2) at (3,3)  {$x_2$};
  \node (x3) at (9,3) {$x_3$};
  \node (y1) at (6,12) {$y_1$};
  \node (y2) at (0,0) {$y_2$};
  \node (y3) at (12,0) {$y_3$};
  \foreach \from/\to in {x0/x1, x0/x2,x0/x3, x1/y1, x2/y2, x3/y3, y1/y2, y1/y3, y2/y3}
    \draw (\from) -- (\to);
\end{tikzpicture}
\end{minipage}
\end{center}
\caption{The poset $\mathcal{F}(M)$ and the complexes $\indepsets{M}, \Cone(\Berg{M}),$ and $\AugBerg{M}$ from Example~\ref{uniform-matroid-example}.}
\label{exampleofM1}
\end{figure}

The augmented Bergman complex has homology $\tilde{H}^1(\AugBerg{M},\Z) \cong \Z^3$, with
$\Z$-basis $\{ [y_1,y_2], [y_1,y_3], [y_2,y_3]\}$, indexed by
the three bases $\bases{M}$;  here one should interpret
oriented edges as
$[y_i,y_j]=-[y_j,y_i]$. The matroid $M$ has automorphism group $G=\mathfrak{S}_3$ acting on the vertices via $g(y_i)=y_{g(i)}$.
Corollary~\ref{homology-rep-theorem} then tells us,
for example, that the transposition $g=(1,3)$ in $\mathfrak{S}_3$ acts on $\tilde{H}_1(\Delta,\Z)$ sending 
$$
\begin{aligned}
g([y_1,y_3])&=[y_3,y_1]=-[y_1,y_3],\\
g([y_1,y_2])&=[y_3,y_2]=-[y_2,y_3].
\end{aligned}
$$

\end{example}


\begin{example}
    The {\it Boolean matroid} $M$ on $E=\{1,2,\ldots,n\}$ of rank $n$ has
    only one basis, $E$ itself, that is, $\bases{M}=\{E\}$. Here
    the augmented Bergman complex $\AugBerg{M}$ triangulates an $(n-1)$-sphere that turns out to be isomorphic to the boundary complex of an $n$-dimensional convex polytope known as the {\it stellohedron}; see \cite[Footnote 7]{BHMPW1} and \cite[\S 10.4]{PostnikovReinerWilliams}. The examples with $n=2,3$ are depicted in Figure~\ref{stellohedra-figure}.

\begin{figure}[h]
\begin{center}
\quad\quad
\begin{minipage}{.2\textwidth}
\begin{tikzpicture}
  [scale=.33,auto=left,every node/.style={circle,fill=black!20}]
  \node (y1) at (2,0) {$y_1$};
  \node (y2) at (6,0) {$y_2$};
  \node (x1) at (0,3) {$x_1$};
   \node (x0) at (4,5) {$x_\varnothing$};
    \node (x2) at (8,3) {$x_2$};
  \foreach \from/\to in {y1/y2, y1/x1, y2/x2, x1/x0, x0/x2}
    \draw (\from) -- (\to);
\end{tikzpicture}
\end{minipage}
\qquad\qquad
\begin{minipage}{.2\textwidth}
\begin{tikzpicture}
  [scale=.33,auto=left,every node/.style={circle, minimum size=0.3cm, fill=black!20}]
    \node (y1) at (6.5,11) {$y_1$};
    \node (y2) at (3,6.5) {$y_2$};
    \node (y3) at (9,6.5) {$y_3$};
    \node (x0) at (6,8) {$x_\varnothing$};
    \node (x1) at (5,15) {$x_1$};
    \node (x2) at (0,3) {$x_2$};
    \node (x3) at (12,3) {$x_3$};
    \node (x12) at (0,11) {$x_{12}$};
    \node (x13) at (12,11) {$x_{13}$};
    \node (x23) at (6,0) {$x_{23}$};
  \foreach \from/\to in {y1/y2, y1/y3, y2/y3,x1/x12,x1/x13,x2/x12,x2/x23,x3/x13,x3/x23,x1/y1,x2/y2,x3/y3,y1/x12,y1/x13,y2/x12,y2/x23,y3/x13,y3/x23}
    \draw (\from) -- (\to);
    \foreach \from/\to in {x0/x12,x0/x13,x0/x23,x0/x1,x0/x2,x0/x3}
    \draw[dashed] (\from) -- (\to);
\end{tikzpicture}
\end{minipage}
\end{center}
\caption{The complex $\AugBerg{M}$ for a rank $n$ Boolean matroid, where $\AugBerg{M}$ is
the boundary of the $n$-dimensional stellohedron.
For $n=2$, the stellohedron is a pentagon, shown at left. For $n=3$ the stellohedron has $16$ boundary triangles (not shaded here), shown at right.}
\label{stellohedra-figure}
\end{figure}
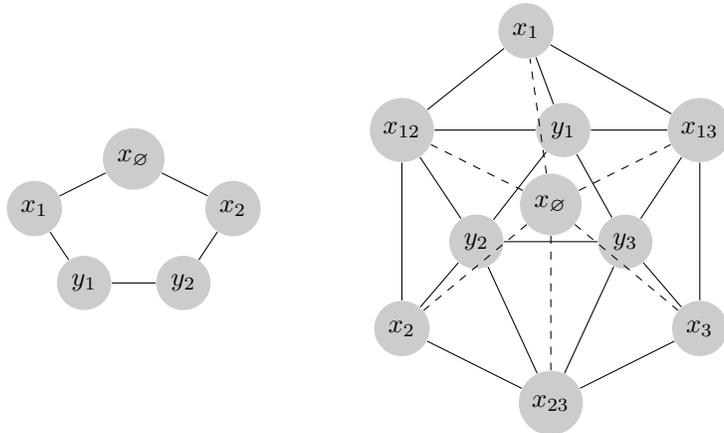

\noindent
Since there is only one basis for $M$, 
Corollary~\ref{homology-rep-theorem}
implies that $\tilde{H}_{n-1}(\AugBerg{M},\Z) \cong \Z$,
with $\Z$-basis element the oriented simplex $[E]=[y_1,\ldots,y_n]$; this corresponds to the homology orientation
class of the boundary sphere of
the stellohedron. 
Here $\Aut(M)$ is the symmetric group permuting $\mathfrak{S}_n$, and it
    acts via the sign representation:  $g([E])=\sgn(g) \cdot [E]$ 
    for every permutation $g$ in $\mathfrak{S}_n$.
\end{example}

\section{Background}
\label{background-section}

\subsection{Matroids}
\label{matroid-subsection}
We begin with two axiomatizations of matroids;
see Oxley \cite[Chap. 1]{Oxley} for other axioms.

\begin{definition}({\it Matroids defined by independent sets})
\label{independent-set-axioms-definition}
A {\it matroid} $M$ on ground set $E$ is a collection
$\indepsets{M} \subseteq 2^E$, called its 
{\it independent sets}, satisfying axioms:
\begin{itemize}
    \item[{\sf I1.}] $\varnothing \in \indepsets{M}$.
    \item[{\sf I2.}] $I \subseteq J$ and $J \in \indepsets{M}$ implies $I \in \indepsets{M}$.
    \item[{\sf I3.}] If $I,J \in \indepsets{M}$ and $\#I <
    \#J$, then there exists $j \in J \setminus I$ with
    $I \cup \{j\} \in \indepsets{M}$.
\end{itemize}
\end{definition}

\noindent
From Axioms {\sf I1, I2}, one sees that the collection $\indepsets{M}$ forms an abstract simplicial complex on vertex set $E$. Axiom {\sf I3} shows that it
is pure of dimension $r-1$ where $r=r(M)$
is the cardinality of all maximal independent sets $B$, called the {\it bases} $\bases{M}$.

Equivalently, one can define a matroid via flats.

\begin{definition}({\it Matroids defined by flats})
\label{flat-axioms-definition}
A {\it matroid} $M$ on ground set $E$ is a collection
$\flats{M} \subseteq 2^E$, called its 
{\it flats}, satisfying axioms:
\begin{itemize}
    \item[{\sf F1.}] $E \in \flats{M}$.
    \item[{\sf F2.}] $F,G \in \flats{M}$ implies $F \cap G \in \flats{M}$.
    \item[{\sf F3.}] For every $F \in \flats{M}$ and
    $e \in E \setminus F$, there exists a unique $G \in \flats{M}$ containing $e$ that {\it covers} $F$ in this sense: there does not exist $H \in \flats{M}$
        with $F \subsetneq H \subsetneq G$.
\end{itemize}
\end{definition}

Assume from here on that $E$ is {\bf finite}. Then
the flats $\flats{M}$ let one define the {\it matroid closure} operator
\begin{equation}
    \label{matroid-closure}
\begin{array}{rcl}
2^E &\longrightarrow &2^E \\
A & \longrightarrow & \overline{A}
\end{array}
\end{equation}
where $\overline{A}$ is the smallest flat containing $A$, namely
$$
\overline{A}:= \bigcap_{\substack{F \in \flats{M}:\\ F \supseteq A}} F.
$$
We will also wish to view $\flats{M}$ as a partially ordered set ({\it poset}) via inclusion, with unique bottom element $\overline{\varnothing}$ and top element $E$. In fact, $\flats{M}$ is a {\it ranked lattice}, with
rank function $r: \flats{M} \rightarrow \{0,1,2,\ldots,r\}$ satisfying $r(\overline{\varnothing})=0$, $r(E)=r$, and
$r(G)=r(F)+1$ if there is no
flat $H$ satisfying $F \subsetneq H \subsetneq G$.

Passing between the independent sets and flats of a matroid $M$ is not hard. First, given any subset $A \subseteq E$, one can define its
{\it rank function} $r(A)$ either using $\indepsets{M}$ or $\flats{M}$ as follows:
$$
\begin{aligned}
r(A)&=\max\{ \#I : I \in \indepsets{M}, I \subseteq A\}\\
&=r(\overline{A}).
\end{aligned}
$$
Then one can recover either of $\indepsets{M}$ or $\flats{M}$ from the rank function $r$ as follows: 
$$
\begin{aligned}
\indepsets{M}&:=\{I \subseteq E: r(I)=\#I\},\\
\flats{M}&:=\{F \subseteq E: r(F) < r(F \cup \{e\}) \text{ for all }e \in E \setminus F\}
\end{aligned}
$$

The following two matroid constructions will turn out to be useful in the sequel.
\begin{definition}
Given a matroid $M$ on ground set $E$, and a subset $A \subseteq E$, one can define two new matroids, the {\it restriction} $M|_A$, and the {\it contraction} $M/A$ as follows:
$$
\begin{aligned}
\indepsets{M|_A}&:=\{ I \in \indepsets{M}: I \subseteq A \} = \indepsets{M} \cap 2^A,\\
\flats{M/A}&:=\{ F \setminus A: F \in \flats{A}, F \supseteq A\}.
\end{aligned}
$$
\end{definition}
\noindent
In particular, for a flat $F$, one has a poset isomorphism between  $\flats{M/F}$ and the poset interval
$$
[F,E]:=\{G \in \flats{M}: F \subseteq G \subseteq E\}.
$$
The isomorphism sends a flat $G$ in the interval $[F,E]$ of $\flats{M}$ to the flat $G\setminus F$ in
$\flats{M/F}$.

\subsection{Order complexes}
\label{order-complex-subsection}

Several simplicial complexes that we will consider come from this construction.

\begin{definition}
Given any partially ordered set $P$, its
{\it order complex} $\Delta P$ is the
abstract simplicial complex on vertex set $P$,
whose simplices are the totally ordered subsets of $P$.
\end{definition}

An element of the poset $P$ that is comparable to all others (such as a least element or greatest element) will give rise to a cone vertex in $\Delta P$, so that $\Delta P$ is contractible. For this reason, such elements are often removed before forming the order complex. For example, since the poset of flats $\flats{M}$ has
least element $\overline{\varnothing}$ and greatest element $E$, they are removed before forming the {\it Bergman complex}
$$
\Berg{M}:=\Delta( \flats{M} \setminus \{ \overline{\varnothing}, E \}).
$$
However, we sometimes put back in the bottom
element $\overline{\varnothing}$ to consider the
cone over $\Berg{M}$, which we denote 
$$
\Cone(\Berg{M}):=\Delta( \flats{M} \setminus \{ E \}.
$$

We recall here from the Introduction that both
$\Berg{M}$ and $\indepsets{M}$ 
are subcomplexes of the following complex, introduced recently in the work of Braden, Huh, Matherne, Proudfoot and Wang \cite[Def. 2.2]{BHMPW1}, and central to their further work \cite{BHMPW2}.

\begin{definition}
Given a matroid $M$ on ground set $E$,
the {\it augmented Bergman complex} $\AugBerg{M}$ is
the abstract simplicial complex on
vertex set 
$$
\{y_i: i \in E \} \sqcup \{x_F: \text{proper flats }F \in \flats{M} \setminus \{E\} \}
$$
whose simplices are the subsets
\begin{equation}
\{y_i\}_{i \in I} 
\sqcup
\{x_{F_1},x_{F_2},\ldots,x_{F_\ell} \}
\end{equation}
for which $I \in \indepsets{M}$
and the (possibly empty) proper flats $F_i$ satisfy
$
I \subseteq 
 F_1 \subsetneq 
  F_2 \subsetneq \cdots
  \subsetneq F_\ell \quad (\subsetneq E).
$
\end{definition}

\subsection{Purity, gallery-connectedness, shellability}
\label{pure-gallery-connected-shellable-section}

We recall a hierarchy of simplicial complex properties.
 
\begin{definition}
A {\it facet} in a simplicial complex $\Delta$ 
is a face which is maximal under inclusion. 

One says that $\Delta$ is {\it pure} and $d$-dimensional if
all of its {\it facets} have
the same cardinality $d+1$.
\end{definition}

\begin{definition}
A pure $d$-dimensional simplicial complex $\Delta$ is {\it gallery-connected}\footnote{Also known as {\it connected in codimension one} or {\it strongly connected} or {\it dually connected}.},
if for any two facets $\phi,\phi'$ one has
a sequence of facets $\phi=\phi_0,\phi_1,\ldots,\phi_t=\phi'$
with $\dim(\phi_i \cap \phi_{i-1})=d-1$ for each $i=1,2,\ldots,t$.
\end{definition}

\begin{definition}
\label{shelling-definition}
A pure $d$-dimensional simplicial complex $\Delta$ is {\it shellable}\footnote{Here we restrict our shellable complexes to be pure; see Bj\"orner and Wachs \cite{BjornerWachs}
for the generalization to nonpure complexes.} if one can order its
facets $\phi_1,\phi_2,\phi_3,\ldots$ in a {\it shelling order}: for all $j \geq 2$, the intersection 
of the subcomplex generated by $\phi_j$ and
the subcomplex generated by all previous facets
$\{ \phi_1,\ldots,\phi_{j-1} \}$
is a pure subcomplex of dimension $d-1$ inside $\phi_j$.
Here is a useful equivalent way to say that a total
ordering $\prec$ on the facets of $\Delta$  
is a shelling order:
$$
\text{for all facets }\phi \prec \phi',\text{ there exists }\phi'' \prec \phi'\text{ such that }
\phi \cap \phi' \subseteq \phi'' \cap \phi'
\text{ and }
\#\phi'' \cap \phi'=\#\phi'-1.
$$
\end{definition}

\medskip
Shellability of $\Delta$ determines
the homotopy type of its {\it geometric realization} $\Vert \Delta \Vert$;  see Kozlov \cite[Chap. 12]{Kozlov}.

\begin{definition}
    Let $\Delta$ be a shellable simplicial complex with shelling order $\phi_1,\phi_2,\phi_3,\ldots$ on its facets. The \textit{restriction face} of facet $\phi_j$ is its subface $\mathscr{R}(\phi_j)$ containing these vertices:
    \[
        \mathscr{R}(\phi_j)=\{x\in \phi_j : \text{ there exists }i\text{ with }1 \leq i < j\text{ and } \phi_j\setminus\{x\}\subset \phi_i\}.
    \]
Call $\phi_j$ a {\it homology facet} in the shelling if $\mathscr{R}(\phi_j)=\phi_j$.
\end{definition}

\begin{lemma}
\label{shellable-homotopy} \cite[Thm. 12.3]{Kozlov}
If a pure $d$-dimensional shellable complex $\Delta$,
has a shelling with exactly $\beta$ homology facets, then its geometric realization is homotopy equivalent to a $\beta$-fold wedge of $d$-spheres. 
\end{lemma}

\noindent
Intuitively, each homology facet in the shelling
``caps off" a $d$-sphere. Furthermore, the
subcomplex obtained by removing all homology facets is contractible, as it is a shellable complex with no homology
facets.

\begin{example}
\label{known-shellings-example}
As mentioned in the Introduction,
$\indepsets{M}$ and $\Berg{M}$ are shellable,
and
their shellings have $T_M(0,1), T_M(1,0)$ homology facets, respectively \cite[\S 7.3, 7.4, 7.6]{Bjorner-chapter},
where $T_M(x,y)$ is the {\it Tutte polynomial}.
\end{example}

\section{Proof of Theorem~\ref{shelling-theorem}}
\label{shellability-section}
We recall here the statement of the theorem.
\vskip.1in
\noindent
{\bf Theorem~\ref{shelling-theorem}.}
{\it
The augmented Bergman complex $\AugBerg{M}$ of a matroid is shellable,
via two families of shellings:
\begin{itemize}
    \item [(i)] one family that shells the facets of $\Cone(\Berg{M})$
    first, and the facets of $\indepsets{M}$
     last, 
       \item [(ii)] one family that shells the facets of $\indepsets{M}$ first, and the facets of $\Cone(\Berg{M})$  last.
\end{itemize}
}

\medskip
\noindent
Before proving it, we identify and conveniently index
facets of $\AugBerg{M}$. Recall from \eqref{typical-augberg-simplex} that faces
of $\AugBerg{M}$ are  
$$
\phi=\{y_i\}_{i \in I} \cup \{x_{F_j}\}_{j=1}^\ell
$$
where $I \in \indepsets{M}$, each $F_j \in \flats{M} \setminus \{E\}$, and
$I \subseteq F_1 \subsetneq \cdots \subsetneq F_\ell$.
This face $\phi$ is a facet if and only if both 
\begin{itemize}
    \item 
$\overline{I}=F_1$ \quad (else one could add the vertex $x_{\overline{I}}$ to $\phi$), and
\item $\#I + \ell=r(M)$ \quad  (else  $\overline{I}=F_1 \subsetneq \cdots \subsetneq F_\ell \subsetneq E$ is a non-maximal chain in interval
$[F_1,E]$ of $\flats{M}$).
\end{itemize}
In comparing facets $\phi,\phi'$, with $\phi$ as above, and 
$
\phi'=\{y_i\}_{i \in I'} \cup \{x_{F'_j}\}_{j=1}^{\ell'},
$
we will use this abbreviated notation: 
letting $F_\subdot$ denote the chain of flats $F_1 \subsetneq \cdots \subsetneq F_\ell$, and similarly for $F'_\subdot$, write
\begin{equation}
\label{typical-pair-of-facets}
\begin{aligned}
\phi &\leftrightarrow (I,F_\subdot),\\
\phi'&\leftrightarrow (I',F'_\subdot).
\end{aligned}
\end{equation}
In our proofs that various linear orders  $\prec$ on the facets of $\AugBerg{M}$ are shellings, as in Definition~\ref{shelling-definition},
we will be given such a pair $\phi, \phi'$
as in \eqref{typical-pair-of-facets} with $\phi \prec \phi'$,
and need to {\bf construct}
\begin{equation}
\label{shelling-condition}
\phi'' \leftrightarrow (I",F''_\subdot)
\text{ for which }
\phi \cap \phi' \subseteq \phi'' \cap \phi'
\text{ and }
\#\phi'' \cap \phi'=\#\phi'-1.
\end{equation}

\subsection{Flag-to-basis shellings}

\begin{definition}
\label{flag-to-basis-order-definition}
Call a total order $\prec$ on the facets $\phi$ of
$\AugBerg{M}$ a {\it flag-to-basis ordering}
if two facets as in \eqref{typical-pair-of-facets}
have $\phi \prec \phi'$
whenever {\it either} of these conditions hold:
\begin{itemize}
    \item[(a)] $\#I < \#I'$, or 
    \item[(b)] $I=I'$, so  $F_1=\overline{I}=\overline{I}'=F'_1$, and
    $F_\subdot$ strictly precedes $F'_\subdot$ in some chosen
    shelling of $\Berg{M/F_1}$.
\end{itemize}
Note that in condition (b), we are identifying the flats
$\flats{M/F_1}$ with the poset interval $[F_1,E]$
in $\flats{M}$.
\end{definition}

\begin{proof}[Proof of Theorem~\ref{shelling-theorem}(i).]
We check that any flag-to-basis ordering $\prec$ on
the facets of $\AugBerg{M}$ gives a shelling. 
Note Definition~\ref{flag-to-basis-order-definition}(a) ensures that $\prec$ orders facets of $\Cone(\Berg{M})$ first
and those of $\indepsets{M}$ last. To check it is a shelling, given facets $\phi \prec \phi'$ as in \eqref{typical-pair-of-facets}, there are two cases
to consider. 

\medskip
\noindent
{\sf Case 1}: $I=I'$. In this case, the shelling of $\AugBerg{M/F_1}$ from Definition~\ref{flag-to-basis-order-definition}(b)
provides the existence of a maximal chain $F''_\subdot$ shelled earlier than $F'_\subdot$
in $\flats{M/F_1}$, and having $F_\subdot \cap F'_\subdot \subseteq F''_\subdot \cap F'_\subdot$
and $\#F''_\subdot \cap F'_\subdot=\#F'_\subdot-1$.
Thus taking $\phi'' \leftrightarrow (I,F''_\subdot)$ does the job
for \eqref{shelling-condition}.

\medskip
\noindent
{\sf Case 2}: $I\neq I'$. Consider the independent set
$I \cap I' \subsetneq I'$, and use Axiom (I3) repeatedly
to find $I'' \in \indepsets{M}$ having
$
I \cap I' \subseteq I'' \subsetneq I'
$
with $\#I''=\#I'-1$. Now let $F''_\subdot:=\{\overline{I''}\} \cup F'_\subdot$.
One can then check that $\phi'' \leftrightarrow (I'',F_\subdot'')$ has $\phi'' \prec \phi'$ (since $\#I'' < \#I')$, and does the job for \eqref{shelling-condition}.
\smallskip
\end{proof}

\subsection{Basis-to-flag shellings}

\begin{definition}
\label{basis-to-flag-order-definition}
Call a total order $\prec$ on the facets $\phi$ of
$\AugBerg{M}$ a {\it basis-to-flag ordering}
if two facets as in \eqref{typical-pair-of-facets}
have $\phi \prec \phi'$
whenever any of these conditions (a),(b), or (c) hold:
\begin{itemize}
    \item[(a)] $\#I > \#I'$, or 
    \item[(b)] $F_1=F_1'$ and
    $F_\subdot$ strictly precedes $F'_\subdot$ in some chosen shelling of $\Berg{M/F_1}$, or
    \item[(c)] $F_\subdot=F'_\subdot$ (so that $\overline{I}=F_1=F_1'=\overline{I'}$), and $I$ strictly precedes $I'$ in some chosen shelling of $\indepsets{M|_F}$. 
\end{itemize}
As before, in condition (b), we 
identify $\flats{M/F_1}$ with the interval $[F_1,E]$
in $\flats{M}$, but now in condition (c), we also
identify $\indepsets{M|_F}$ with $\indepsets{M} \cap 2^{F_1}$.
\end{definition}

\begin{proof}[Proof of Theorem~\ref{shelling-theorem}(ii).]
We check that any basis-to-flag ordering $\prec$ on
the facets of $\AugBerg{M}$ gives a shelling. 
Note Definition~\ref{basis-to-flag-order-definition}(a) ensures that $\prec$ orders facets of $\indepsets{M}$ first
and those of $\Cone(\Berg{M})$ last. To check it is a shelling, given facets $\phi \prec \phi'$ as in \eqref{typical-pair-of-facets}, there are three cases to consider. 

\medskip
\noindent
{\sf Case 1}: $F_1=F_1'$, but
$F_\subdot \neq F'_\subdot$.

Then Definition~\ref{basis-to-flag-order-definition}(c) ensures that
$F_\subdot$ strictly precedes $F'_\subdot$ in our chosen shelling of $\Berg{M/F_1}$. As before, this shelling
of $\Berg{M/F_1}$ provides the existence of a maximal chain $F''_\subdot$
in $\flats{M/F_1}$ having $F_\subdot \cap F'_\subdot \subseteq F''_\subdot \cap F'_\subdot$
and $\#F''_\subdot \cap F'_\subdot=\#F'_\subdot-1$.
Taking $\phi'' \leftrightarrow (I,F''_\subdot)$ does the job
for \eqref{shelling-condition}.

\medskip
\noindent
{\sf Case 2}: $F_\subdot= F'_\subdot$.

Here the shelling in 
Definition~\ref{basis-to-flag-order-definition}(c)
provides the existence of an independent set $I''$ with $\overline{I''} = F_1$ having $I \cap I' \subseteq I'' \cap I'$
and $\#I'' \cap I'=\#I'-1$.
Thus taking $\phi'' \leftrightarrow (I'',F_\subdot)$ does the job for
\eqref{shelling-condition}.

\medskip
\noindent
{\sf Case 3}: $F_1\neq F_1'$.

In this case, choose any element $i_0 \in F_2' \setminus F_1'$, and then
define $I'':=I' \cup \{i_0\}$ and $F''_\subdot:=F'_\subdot \setminus \{F_1'\}$. 
One can then check that $\phi'' \leftrightarrow (I'',F_\subdot'')$ has $\phi'' \prec \phi'$ (since $\#I'' > \#I'$), and does the job for \eqref{shelling-condition}.
\end{proof}

\section{Proof of Corollary~\ref{homotopy-type-corollary}}
\label{homotopy-type-section}

We recall here the statement of the corollary.

\vskip.1in
\noindent
{\bf Corollary~\ref{homotopy-type-corollary}.}
{\it
For a matroid $M$, the augmented Bergman complex $\AugBerg{M}$
is homotopy equivalent to a $\beta$-fold wedge of
$(r(M)-1)$-spheres, with two different expressions for $\beta$:
\begin{align*}
\beta&=\#\bases{M}\quad (=T_M(1,1))\\
\beta&=\sum_{F \in \flats{M}} T_{M|_F}(0,1) \cdot T_{M/F}(1,0).
\end{align*}
The first and second expressions for $\beta$ above 
(labeled \eqref{flag-to-basis-expression-for-beta}, \eqref{basis-to-flag-expression-for-beta} in the Introduction)
are predicted by the shellings in 
Theorem~\ref{shelling-theorem}(i),(ii),
respectively.
}

\begin{proof}
Recall $\beta$ counts 
homology facets $\phi$, that is, those with $\mathscr{R}(\phi)=\phi$,
for the shellings in Theorem~\ref{shelling-theorem}(i),(ii).

\medskip
\noindent
{\sf Proof of \eqref{flag-to-basis-expression-for-beta}}.
Assume that $\prec$ is a flag-to-basis shelling order on the facets, as in Theorem~\ref{shelling-theorem}(i).
We will show that $\phi \leftrightarrow (I,F_\subdot)$ is a homology facet if and only if $F_\subdot=\varnothing$,
that is, $I$ is a basis. 

For the ``if" direction, assume $I=\{b_1,\ldots,b_r\}$ is a basis and $F_\subdot=\varnothing$, so $\phi=\{y_{b_1},\ldots,y_{b_r}\}$. Then 
every vertex $y_{b_i}$ lies in  $\mathscr{R}(\phi)$, since
$\phi\setminus \{y_{b_i}\} \subset \phi' \prec \phi$
where $\phi' \leftrightarrow (I',F_\subdot')$
with $I':=I \setminus \{b_i\}$ and $F'_\subdot:=\{\overline{I'}\}$.
The fact that $\phi' \prec \phi$
uses Definition~\ref{flag-to-basis-order-definition}(a).

For the ``only if" direction, assume that $F_\subdot =\{x_{F_1},\ldots,x_{F_\ell}\} \neq\varnothing$, and we will show that $\mathscr{R}(\phi) \neq \phi$ because $x_{F_1} \not\in \mathscr{R}(\phi)$.
To see this, note that any facet $\phi'\leftrightarrow (I',F'_\subdot)$ 
containing $\phi \setminus \{x_{F_1}\}$
must either have $\#I'=\#I+1$ (so $\phi' \succ \phi$),
or have $I'=I$ and hence  $F'_1=\overline{I'}=\overline{I}=F_1$,
which forces $\phi=\phi'.$

\medskip
\noindent
{\sf Proof of \eqref{basis-to-flag-expression-for-beta}}.

Assume that $\prec$ is a basis-to-flag shelling order on the facets, as in Theorem~\ref{shelling-theorem}(ii).
We will show that $\phi \leftrightarrow (I,F_\subdot)$ is a homology facet if and only if it satisfies the following conditions:
considering the flat $F=\overline{I}$ (possibly $F=E$ when $I$ is a basis; otherwise $F=F_1$) one has
both that
\begin{itemize}
    \item[(I)] $F_\subdot$ is a homology facet in the chosen shelling
for $\Berg{M/F}$, and 
\item[(II)] $I$ is a homology facet in the chosen shelling
for $\indepsets{M|_F}$. 
\end{itemize}
This would prove \eqref{basis-to-flag-expression-for-beta},
since then the homology facets for the order $\prec$ would
be parametrized as follows: first choose the flat $F$ in $\flats{M}$ arbitrarily, then choose $F_\subdot$ from one of $T_{M/F}(1,0)$ choices, and lastly choose $I$
independently from one of $T_{M|F}(0,1)$ choices; see
Example~\ref{known-shellings-example}

To check that conditions (I),(II) indeed describe the homology facets for the
$\prec$ shelling order, first deal with the special case
when $F:=\overline{I}=E$, so that $I$ is a basis and $F_\subdot=\varnothing$.
It was already noted that Definition~\ref{basis-to-flag-order-definition}(a) implies
$\prec$ shells the facets of
the subcomplex $\indepsets{M}$ first. Hence $I$ will be a homology facet for the $\prec$ shelling if and only if it is a homology facet for the chosen shelling of $\indepsets{M}=\indepsets{M|_E}$, as in condition (II). Since
$F_\subdot=\varnothing$, condition (I) 
above is vacuously satisfied in this case.

Now assume we are in the more 
generic case, when $F:=\overline{I}=F_1\neq E$. We need to understand whether or not a typical vertex $x$ of $\phi$ 
lies in $\mathscr{R}(\phi)$, 
that is, whether $\phi\setminus\{x\}$ lies in some earlier facet $\phi' \prec \phi$.
There are three cases to consider for $x$.

\medskip
\noindent
{\sf Case 1.} $x=x_{F_1}$.
In this case, one {\it always} has $x \in\mathscr{R}(\phi)$.
To see this, pick any $i_0 \in F_2 \setminus F_1$, and define
the facet
$\phi'\leftrightarrow (I',F_\subdot')$ where 
$I':=I \cup \{i_0\}$ and $F_\subdot':=F_\subdot \setminus \{F_1\}$. Then $\phi' \supseteq \phi \setminus \{x\}$
and $\phi' \prec \phi$ since $\#I' > \#I$.
    
\medskip
\noindent
{\sf Case 2.} $x=x_{F_j}$ for $j \geq 2$.
In this case, Definition~\ref{basis-to-flag-order-definition}(b)
shows that $x$ lies in $\mathscr{R}(\phi)$ if and only
if the vertex $x_{F_j \setminus F_1}$ lies in $\mathscr{R}(F_\subdot)$ in the chosen
shelling for $\Berg{M/F_1}$. 

\medskip
\noindent
{\sf Case 3.} $x=y_i$ for some $i \in I$.
In this case, Definition~\ref{basis-to-flag-order-definition}(c) shows that
$x$ lies in $\mathscr{R}(\phi)$ if and only if
$y_i$ lies in $\mathscr{R}(I)$ in the chosen
shelling for $\indepsets{M|_{F_1}}$.

\medskip
Hence conditions (I),(II) above characterize
homology facets for the $\prec$ shelling, completing the proof.
\end{proof}

\begin{remark}
Tutte's original definition of the Tutte polynomial $T_M(x,y)$
involved choosing a linear order $\omega$ on the ground set $E$. From this he defined for each basis
$B$ in $\bases{M}$ its {\it internal activity} $i_\omega(B)$ and {\it external activity} $e_\omega(B)$
with respect to $\omega$, and then one has \cite[\S 7.3,eqn (7.11)]{Bjorner-chapter}
$$
T_M(x,y) = \sum_{B \in \bases{M}} x^{i_\omega(B)} y^{e_\omega(B)}.
$$
In particular, $T_M(0,1)$ and $T_M(1,0)$
count the bases with internal and external activity zero, respectively. One can choose shelling orders for $\indepsets{M}$
and $\Berg{M}$ having homology facets indexed by
such bases; see Bj\"orner \cite[\S 7.3, 7.6]{Bjorner-chapter}. Consequently, one can choose the basis-to-flag shellings in Theorem~\ref{shelling-theorem}(ii)
so that their homology facets 
are indexed by triples $(F,I,I')$ that
combinatorially interpret the right side of
\eqref{basis-to-flag-expression-for-beta}:
\begin{itemize}
    \item  $F$ is a flat,
    \item $I$ a basis for $M|_F$ with
internal activity zero, and 
\item $I'$ a basis of
$M/F$ with external activity zero.
\end{itemize}
Bijections between the set $\bases{M}$ and the set of triples $\{ (F,I,I') \}$
as above appear in \cite{EtienneLasVergnas, KookReinerStanton}.

\end{remark}

\section{Proof of Corollary~\ref{homology-rep-theorem}}
\label{equivariant-section}

We recall the statement of the corollary.

\vskip.1in
\noindent
{\bf Corollary~\ref{homology-rep-theorem}.}
{\it
The action of $\Aut(M)$ on the top (reduced)
homology $\tilde{H}_{r(M)-1}(\AugBerg{M},\Z)$
is the same as its action on the top
oriented simplicial chain group $\tilde{C}_{r(M)-1}(\indepsets{M},\Z)$ for the
complex $\indepsets{M}$.
}

\medskip
\noindent
Recall that one can compute (reduced) simplicial homology $\tilde{H}_*(\Delta,\Z)$
for a simplicial complex
$\Delta$ using {\it oriented simplicial chains}; see, e.g., Munkres \cite[\S1.5]{Munkres}. The
$i^{th}$ chain group $\tilde{C}_d(\Delta,\Z)$ has the following description. Fix for each $i$-dimensional simplex $\sigma$ having vertex set $\{v_0,v_1,\ldots,v_i\}$
a {\it reference ordering} $(v_0,v_1,\ldots,v_i)$, and then $\tilde{C}_i(\Delta,\Z)$ is
a free abelian group having one $\Z$-basis element
$[v_0,v_1,\ldots,v_i]$, called an {\it oriented simplex}, for each such $\sigma$, and for any permutation $w$
in the symmetric group $\mathfrak{S}_{i+1}$, one sets
$$
[v_{w(0)},v_{w(1)},\ldots,v_{w(i)}] :=
\sgn(w) \cdot [v_0,v_1,\ldots,v_i]
$$
where $\sgn(w) \in \{+1,-1\}$ is the usual {\it sign} of the permutation $w$.

\medskip
\noindent
We claim that Corollary~\ref{homology-rep-theorem} will
be another consequence
of the flat-to-basis shellings of $\AugBerg{M}$ from
Theorem~\ref{shelling-theorem}(i),
similar to equation
\eqref{flag-to-basis-expression-for-beta}. The essential point is that matroid automorphisms permute the bases $\bases{M}$,
which index the homology facets
for these shellings. In fact, we will deduce Corollary~\ref{homology-rep-theorem} from a lemma that applies to a slightly more general notion of homology facets.

\begin{definition}
In a simplicial complex $\Delta$, call a collection of its facets $\cB$ a set of
{\it homology facets} if the subcomplex $\Delta \setminus \cB$ obtained by removing them is contractible.
\end{definition}

This leads to the following generalization of
Lemma~\ref{shellable-homotopy}.

\begin{lemma}
\label{homology-facets-lemma}
When $\Delta$ has a collection $\cB$ of homology facets, it is homotopy equivalent to a wedge of spheres: 
$$
\Delta \approx\bigvee_{\sigma \in \cB} \mathbb{S}^{\dim(\sigma)}.
$$
Furthermore, any group $G$ of simplicial automorphisms of $\Delta$ preserving
$\cB$ setwise will act on $\tilde{H}_i(\Delta,\Z)$
via its signed permutation representation on the $\Z$-submodule $\mathrm{span}_\Z\{ [\sigma]: \sigma \in \cB, \dim(\sigma)=i\}.$ within $\tilde{C}_i(\Delta,\Z)$.
\end{lemma}

\begin{remark}
\label{signed-permutation-rep-remark}
More explicitly, if $g(\sigma)=\sigma'$, and
$\sigma, \sigma'$ with $[\sigma]=[v_0,v_1,\ldots,v_i]$
and $[\sigma']=[v_0',v_1',\ldots,v_i']$,
then 
$
g[\sigma]=\pm [\sigma']
$
where the $\pm$ is $\sgn(w)$ for $w$ defined by $(g(v_0),\ldots,g(v_i))=(v_{w(0)}',\ldots,v_{w(i)}')$. One can also view this signed permutation representation 
as a direct sum
$
\bigoplus_{\sigma} \sgn_\sigma \uparrow_{G_\sigma}^G
$
of {\it induced representations}. Here $\sigma$ runs through
any choice of $G$-orbit representatives for $\cB$,
and $G_\sigma$ is the subgroup of $G$  setwise stabilizing the vertex set 
$\{v_0,v_1,\ldots,v_d\}$ of $\sigma$,
with $\sgn_\sigma: G_\sigma \rightarrow \{+1,-1\}$ its sign character.
\end{remark}

\begin{proof}[Proof of Lemma~\ref{homology-facets-lemma}]
(cf. proof of \cite[Theorem 7.7.2]{Bjorner-chapter})
Name $\Delta':=\Delta \setminus \cB$.
Then the homotopy type assertion is a
consequence of what Bj\"orner calls the {\it Contractible Subcomplex Lemma} \cite[Lemma 10.2]{Bjorner-handbook}:  for
a contractible subcomplex 
$\Delta' \subset \Delta$, the
projection $\Vert \Delta \Vert \twoheadrightarrow \Vert \Delta \Vert /\Vert \Delta'\Vert$ is a homotopy equivalence.

For the homology assertion, start with 
the long exact sequence in integral homology for the pair $(\Delta,\Delta')$,
$$
\cdots \rightarrow 
\tilde{H}_i(\Delta') \rightarrow  
\tilde{H}_i(\Delta) \rightarrow 
\tilde{H}_i(\Delta,\Delta') \rightarrow 
\tilde{H}_{i-1}(\Delta') \rightarrow \cdots
$$
Contractibility of $\Delta'$ implies $\tilde{H}_i(\Delta')=0$ for all $i$, giving isomorphisms 
$
\tilde{H}_i(\Delta) \cong 
\tilde{H}_i(\Delta,\Delta').
$
On the other hand, since each simplex in $\cB=\Delta \setminus \Delta'$ 
is a {\it facet} of $\Delta$, lying in
no higher-dimensional faces, the boundary maps in the complex $\tilde{C}_*(\Delta,\Delta')$
computing  $\tilde{H}_*(\Delta,\Delta')$ are all zero. Hence 
$\tilde{H}_i(\Delta,\Delta') = \tilde{C}_i(\Delta,\Delta')$ for all $i$. Furthermore, our assumptions on $G$
imply that all of these isomorphisms commute with the $G$-action. Lastly, note  
 $\tilde{C}_i(\Delta,\Delta)$ has the same
 $\Z$-basis and $G$-action as  $\mathrm{span}_\Z\{ [\sigma] : \sigma \in \cB, \dim(\sigma)=i\}$ within $\tilde{C}_i(\Delta)$.
\end{proof}

\begin{proof}[Proof of Corollary~\ref{homology-rep-theorem}.]

Apply Lemma~\ref{homology-facets-lemma} to
the flat-to-basis shellings from Theorem~\ref{shelling-theorem}(i) of $\AugBerg{M}$. The homology facets are indexed
by the bases $\bases{M}$, and preserved
by the group $G=\Aut(M)$, and all have
dimension $r(M)-1$. Furthermore, note that
within $\tilde{C}_{r(M)-1}(\AugBerg{M},\Z)$,
one has
$$
\mathrm{span}_\Z\{[\sigma]: \sigma \in \bases{M}\}
=\tilde{C}_{r(M)-1}(\indepsets{M},\Z)
$$
since the facets of $\AugBerg{M}$ indexed
by bases of $M$ happen to be exactly the 
facets of $\indepsets{M}$.
\end{proof}

\begin{remark}
Corollary~\ref{homology-rep-theorem} is closely related to 
an identity of representations
from work of Kook, Reiner and Stanton
\cite{KookReinerStanton}
on eigenspaces of {\it combinatorial Laplacians}
for $\indepsets{M}$. Specifically,
taking $i=r(M)-1$ in 
their \cite[Thm. 19]{KookReinerStanton}, 
asserts the following
isomorphism of $G$-representations
for $G=\Aut(M)$:
\begin{equation}
    \label{KRS-eigenspace-rep-decomposition}
\tilde{C}_{r(M)-1}(\indepsets{M})
\cong \bigoplus_{F \in \flats{M}}
\left[ 
\tilde{H}_{r(F)-1}(\indepsets{M|_F} )
\,\, \otimes \,\,
\tilde{H}^{r(M/F)-2}(\Berg{M/F} )
\right] \uparrow_{G_F}^G
\end{equation}
where $G_F=\{g \in G: g(F)=F\}$,
and $\left[-\right]\uparrow_{G_F}^G$ denotes
induction of representations from $G_F$ to $G$.

As we have seen, the flat-to-basis shelling in
Theorem~\ref{shelling-theorem}(i) led to Corollary~\ref{homology-rep-theorem},
showing the $G$-action on the left side of 
\eqref{KRS-eigenspace-rep-decomposition}
is the same as the one on $\tilde{H}_{r(M)-1}(\AugBerg{M})$.
Similarly, with a bit more work 
(details omitted here),
one can use the basis-to-flat shelling in
Theorem~\ref{shelling-theorem}(ii),
and its resulting bases for
$\tilde{H}_{r(M)-1}(\AugBerg{M})$ as
in Bj\"orner \cite[Thm. 7.7.2]{Bjorner-chapter}
to show that the $G$-action on 
$\tilde{H}_{r(M)-1}(\AugBerg{M})$
is isomorphic to the direct sum on the right
side of \eqref{KRS-eigenspace-rep-decomposition}.
\end{remark}

\begin{remark}
The aforementioned work \cite{KookReinerStanton} showed a remarkable property for $\indepsets{M}$ and its  simplicial boundary maps $\{ \partial_i \}_{i=1,2,\ldots}$:
their associated {\it combinatorial Laplacian matrices} $\{ \partial_i^T \partial_i \}$ have only {\it integer} eigenvalues. One might therefore ask
whether $\AugBerg{M}$ shares this property.
Sadly, this fails already for the Boolean matroid $M$ of rank $2$, where $\AugBerg{M}$ is the $5$-cycle graph shown in Figure~\ref{stellohedra-figure}. One can check that its Laplacian matrix $\partial_1^T \partial_1$ has characteristic polynomial $x(x^2-5x+5)^2$, whose eigenvalues are not all integers.

\end{remark}

\section{Augmented Bergman complexes for other closures}
\label{other-closures-section}

One can characterize a matroid $M$ on ground set $E$ in terms of its {\it matroid closure operator} $A \longmapsto \overline{A}$ 
defined in \eqref{matroid-closure}. This is an instance of the following more general notion.

\begin{definition}
Given a set $E$, a map $2^E \overset{f}{\longrightarrow} 2^E$
is called a {\it closure operator} on $E$ if it satisfies three axioms: for all subsets $A, B \subseteq E$,
\begin{itemize}
    \item[{\sf C1.}] $A \subseteq f(A)$
    \item[{\sf C2.}] $A \subseteq B$ implies $f(A) \subseteq f(B)$
    \item[{\sf C3.}] $f(f(A))=f(A)$
\end{itemize}
\end{definition}

\noindent
Any closure operator on a finite set $E$ has analogues of the complexes $\indepsets{M}, \Berg{M}, \AugBerg{M}$, introduced next.

\begin{definition}
Given a closure operator $f$ on a finite set $E$, define a subset
$I \subseteq E$ to be {\it independent} if 
$$
f(I \setminus \{i\}) \subsetneq f(I) \text{ for all }i \in I.
$$
Let $\indepsets{f}$ denote the collection of all independent subsets $I \subseteq E$.
It is not hard to check that $\indepsets{f}$ always
satisfies axioms ${\sf I1, I2}$ from Definition~\ref{independent-set-axioms-definition},
so that it defines a simplicial complex,
also denoted $\indepsets{f}$.
\end{definition}

\begin{definition}
Given a closure operator $f$ on a finite set $E$, define its
{\it poset of closed sets} 
$$
\flats{f}:=\{ F \subseteq E: f(F)=F\}
$$
partially ordered via inclusion.
It is not hard to check that $\flats{f}$ always
satisfies axioms ${\sf F1, F2}$ from Definition~\ref{flat-axioms-definition},
so that it becomes a lattice.
Define the {\it Bergman complex} 
$$
\Berg{f}:=\Delta(\,\, 
\flats{f} \setminus\{f(\varnothing),E\}
\,\,)
$$
to be the order complex of the proper part of this lattice $\flats{f}$.
\end{definition}

\begin{definition}
Given a closure operator $f$ on a finite set $E$, define
its {\it augmented Bergman complex} $\AugBerg{f}$ to
be the abstract simplicial complex on
vertex set 
$$
\{y_i: i \in E \} \sqcup \{x_F:  F \in \flats{f} \setminus \{E\} \}
$$
whose simplices are the subsets
\begin{equation}
\label{typical-closure-augberg-simplex}
\{y_i\}_{i \in I} 
\cup
\{x_{F_1},x_{F_2},\ldots,x_{F_\ell} \}
\end{equation}
for which $I \in \indepsets{f}$
and the $F_i$ are all closed sets in $\flats{f}$, satisfying
$
I \subseteq 
 F_1 \subsetneq 
  F_2 \subsetneq \cdots
  \subsetneq F_\ell \quad (\subsetneq E).
$
\end{definition}

As before with matroid closures, $\AugBerg{f}$ always
contains as subcomplexes both 
\begin{itemize}
    \item 
$\indepsets{f}$
as the simplices in
\eqref{typical-closure-augberg-simplex} with $\ell=0$, and
\item $\Cone(\Berg{f})=\Delta(\,\flats{F} \setminus \{E\}\,\,)$ having cone vertex $x_{f(\varnothing)}$, as
the simplices in \eqref{typical-closure-augberg-simplex} with $\#I = 0$.
\end{itemize}

\noindent
However, in contrast to matroid closures, the complexes $\indepsets{f}, \Berg{f}, \AugBerg{f}$ need not be pure, nor shellable.

\begin{example}
\label{closure-example}
Consider the closure operator $f: 2^E \rightarrow 2^E$ with $E=[5]=\{1,2,3,4,5\}$ whose
poset of closed sets $\flats{f}$ is depicted at top
left in Figure~\ref{closure-example-complexes-figure}. One can compute
the closure $f(A)$ as the intersection of all
$F \in \flats{M}$ containing $A$; for example, $f(\{5\})=\{4,5\}$ and $f(\{1,4\})=\{1,2,3,4,5\}$. The complexes $\indepsets{f}$ and $\Berg{f}$ are shown in the
top row, in the middle and at right. The second row depicts the
deletion $\AugBerg{f} \setminus \{x_\varnothing\}$
of the vertex $x_\varnothing$ from the augmented
Bergman complex $\AugBerg{f}$.

\begin{figure}[h] 

\begin{center}
\begin{minipage}{.2\textwidth}
\begin{tikzpicture}
  [scale=.33,auto=left,every node/.style={circle,fill=black!20}]
  \node (n12345) at (6,9) {12345};
  \node (n1) at (0,3)  {1};
  \node (n2) at (4,3)  {2};
  \node (n3) at (8,3)  {3};
   \node (n4) at (12,3)  {4};
     \node (n12) at (0,6)  {12};
  \node (n13) at (4,6)  {13};
  \node (n23) at (8,6)  {23};
   \node (n45) at (12,6)  {45};
  \node (n0) at (6,0) {$\varnothing$};
   \foreach \from/\to in {n12345/n12,n12345/n13,n12345/n23,n12345/n45, 
   n1/n12,n1/n13,n2/n12,n2/n23,n3/n13,n3/n23,n4/n45,
   n1/n0, n2/n0, n3/n0,n4/n0}
    \draw (\from) -- (\to);
\end{tikzpicture}
\end{minipage}
\qquad \qquad \qquad   
\begin{minipage}{.2\textwidth}
\begin{tikzpicture}
  [scale=.33,auto=left,every node/.style={circle,fill=black!20}]
  \draw[fill=gray!50](0.5,3) -- (7.5,3) -- (4,5.5)-- (0.5,3);
  \node (y1) at (0,3) {$y_1$};
  \node (y2) at (8,3) {$y_2$};
  \node (y3) at (4,6) {$y_3$};
  \node (y4) at (3,10) {$y_4$};
  \node (y5) at (3,0) {$y_5$};
\draw (y5) -- (3.5,3);
\draw[dashed](3.5,3) -- (y3);
  \foreach \from/\to in {y4/y1,y4/y2,y4/y3,y5/y1,y5/y2}
    \draw (\from) -- (\to);
\end{tikzpicture}
\end{minipage}
\qquad \qquad 
\begin{minipage}{.2\textwidth}
\begin{tikzpicture}
  [scale=.33,auto=left,every node/.style={circle,fill=black!20}]
  \node (x1) at (4,12) {$x_1$};
  \node (x2) at (8,4) {$x_2$};
  \node (x3) at (0,4) {$x_3$};
  \node (x12) at (8,8) {$x_{12}$};
  \node (x13) at (0,8) {$x_{13}$};
  \node (x23) at (4,0) {$x_{23}$};
  \node (x4) at (12,8) {$x_4$};
  \node (x45) at (12,4) {$x_{45}$};
  \foreach \from/\to in {x1/x12,x1/x13,x2/x12,x2/x23,x3/x13,x3/x23,x4/x45}
    \draw (\from) -- (\to);
\end{tikzpicture}
\end{minipage}
\end{center}

\medskip 

\begin{center}
\begin{minipage}{.2\textwidth}
 \hspace*{-0.6\linewidth}
\begin{tikzpicture}
  [scale=.30,auto=left,every node/.style={circle,fill=black!20}]
  \draw[fill=gray!50](8,8) -- (12,4) -- (4,4)-- (8,8);
   \draw[fill=gray!50](24,8) -- (24,4) -- (16,13)-- (24,8);
  \draw[fill=gray!50](8,12) -- (8,8) -- (16,8)-- (8,12);
  \draw[fill=gray!50](8,12) -- (8,8) -- (0,8)-- (8,12);
  \draw[fill=gray!50](4,4) -- (8,8) -- (0,8)-- (4,4);
  \draw[fill=gray!50](4,4) -- (0,8) -- (0,4)-- (4,4);
  \draw[fill=gray!50](4,4) -- (8,0) -- (0,4)-- (4,4);
  \draw[fill=gray!50](4,4) -- (8,0) -- (12,4)-- (4,4);
  \draw[fill=gray!50](16,4) -- (8,0) -- (12,4)-- (16,4);
  \draw[fill=gray!50](16,4) -- (16,8) -- (12,4)-- (16,4);      
  \draw[fill=gray!50](8,8) -- (16,8) -- (12,4)-- (8,8);
  \node (y1) at (8,8) {$y_1$};
  \node (y2) at (12,4) {$y_2$};
  \node (y3) at (4,4) {$y_3$};
  \node (y4) at (16,13) {$y_4$};
  \node (y5) at (16,-3) {$y_5$};
  \node (x1) at (8,12) {$x_1$};
  \node (x2) at (16,4) {$x_2$};
  \node (x3) at (0,4) {$x_3$};
  \node (x12) at (16,8) {$x_{12}$};
  \node (x13) at (0,8) {$x_{13}$};
  \node (x23) at (8,0) {$x_{23}$};
  \node (x4) at (24,8) {$x_4$};
  \node (x45) at (24,4) {$x_{45}$};
  \foreach \from/\to in {y4/y1,y4/y2,y4/y3,y5/y1,y5/y2,y5/y3,y5/x45}
    \draw (\from) -- (\to);
\end{tikzpicture}
\end{minipage}
\end{center}

\caption{The lattice $\flats{f}$, and the complexes
$\indepsets{f}$ and $\Berg{f}$, along with 
the complex $\AugBerg{f} \setminus \{x_\varnothing\}$ for the closure in Example~\ref{closure-example}.}

\label{closure-example-complexes-figure}
\end{figure}

\end{example}


It is not hard to show that any finite lattice
is isomorphic to $\flats{f}$ for some closure $f$,
and hence we cannot expect to say much about
the homotopy type of the Bergman complex $\Berg{f}$ in general; we expect
that the same holds for the independent set complex $\indepsets{f}$. 

Nevertheless, we claim that one still has the assertion of Corollary~\ref{homology-rep-theorem}
on the topology of the augmented Bergman complex $\AugBerg{f}$, after
appropriately defining {\it bases} and {\it automorphisms} for a closure.

\begin{definition}
For a closure $f$ on a finite set $E$, define the set $\bases{f}$ of {\it bases}
$$
\bases{f}:=\{ B \in \indepsets{f}: f(B)=E\}.
$$
\end{definition}
\begin{definition}
For a closure $f$ on $E$, define its {\it automorphism group} 
$$
\Aut(f):=\{\text{bijections }E \overset{g}{\longrightarrow} E: f(g(A))=g(f(A))
\text{ for all }A \subseteq E\}
$$
\end{definition}

\noindent
Note $\Aut(f)$ stabilizes
$\bases{f}$, and acts on
$\indepsets{f}, \Berg{f}, \AugBerg{f}$
via simplicial automorphisms.

\begin{thm}
\label{closure-augberg-homology-representation-theorem}
For any closure operator on a finite set $E$,
the bases $\bases{f}$ index a collection
of homology facets for 
$\AugBerg{f}$. 
Hence $\Delta$ is homotopy equivalent to a wedge of spheres
$
\bigvee_{B \in \bases{f}} \mathbb{S}^{\#B-1}.
$
and the action of
$\Aut(f)$ on
$\tilde{H}_i(\Delta,\Z)$ is the same
as on  $\mathrm{span}_\Z\{ [B]: B \in \bases{f}, \#B-1=i\}.$ within $\tilde{C}_i(\AugBerg{f},\Z)$.

\end{thm}

\begin{proof}
In light of Lemma~\ref{homology-facets-lemma}, it suffices to show that
$\Delta=\AugBerg{f}$ has its subcomplex $\Delta':=\Delta \setminus \cB$ contractible. Our proof strategy introduces another simplicial complex $\Delta''$, and shows it has these two properties:
\begin{itemize}
    \item[(a)] $\Delta''$ is a simplicial subdivision of $\Delta'$, and hence homeomorphic to it.
    \item[(b)] $\Delta''$ is homotopy equivalent to the subcomplex $\Cone(\Berg{f})=\Delta(\cF \setminus \{E\})$ inside $\AugBerg{f}$.
\end{itemize}
Since cones are contractible, (b) would show  $\Delta''$ is contractible, and then (a) would show the same for $\Delta'$.
We define $\Delta''$ to be the simplicial complex on vertex
set 
$$
\{y_I: I \in \indepsets{f} \setminus \bases{f} \} \cup \{x_F: F \in \cF \setminus \{E\} \}
$$
whose simplices are subsets
$
\{y_{I_1}, \cdots, y_{I_k}\} 
\cup
\{x_{F_1},\ldots,x_{F_\ell} \}
$
with
$
I_1 \subsetneq  \cdots \subsetneq I_k \subseteq 
 F_1 \subsetneq 
   \cdots
  \subsetneq F_\ell (\subsetneq E).
$

\medskip
\noindent
{\sf Proof of assertion (a). }
One can view $\Delta''$ as having been obtained from $\Delta'$ by performing
{\it barycentric subdivision} \cite[\S 2.15]{Munkres} $\sigma \mapsto \mathrm{Sd}(\sigma)$ to every simplex within the subcomplex $\indepsets{f} \setminus \bases{f}$ of $\Delta'$; each simplex $\sigma=\{y_i\}_{i \in I}$ is replaced by
the simplices $\{y_{I_1},\ldots,y_{I_k}\}$
for which
$
I_1 \subsetneq I_2 \subsetneq \cdots \subsetneq I_k \subseteq I.
$
More generally, the typical simplex of $\Delta'$
as in \eqref{typical-closure-augberg-simplex}
is the {\it simplicial join} $\sigma * \sigma'$ 
\cite[\S 8.62]{Munkres}
of the simplex $\sigma=\{y_i\}_{i \in I}$ above and
the simplex $\sigma'=\{x_{F_1},\ldots,x_{F_\ell}\}$;
one replaces this with the simplicial join 
$\mathrm{Sd}(\sigma) * \sigma'$ in $\Delta''$.

\medskip
\noindent
{\sf Proof of assertion (b). }
Attempt to define a simplicial map $\Delta'' \overset{\pi}{\longrightarrow} \Delta(\cF \setminus \{E\})$ via this set map on vertices:
$$
\begin{array}{rcll}
x_F & \longmapsto & x_F& \text{ for }F \in \flats{f} \setminus \{E\} \\
y_I & \longmapsto & x_{f(I)} & \text{ for }I \in \indepsets{f} \setminus \bases{f}\\
\end{array}
$$
It is not hard to check that $\pi$ indeed carries simplices to simplices, that is, it is a well-defined
simplicial map. One can also check that, for every element $F$ in the poset $\flats{f} \setminus \{E\}$, the inverse image under $\pi$ of the
order complex $\Delta \flats{f}_{\leq F}$ of the principal order ideal $\flats{f}_{\leq F}$ is the star of the vertex $x_F$ within $\Delta''$, and hence contractible. Thus by Quillen's Fiber Lemma \cite[(10.5)(i)]{Bjorner-handbook},
the map $\pi$ induces a homotopy equivalence.
\end{proof}

\begin{remark}
Note that $\Cone(\Berg{M})=\Delta(\flats{f} \setminus \{E\})$ can be identified with
the subcomplex of $\Delta''$ induced on the vertex subset $\{x_F: F \in \flats{f} \setminus \{E\}\}$.
Since these vertices $x_F$ are all pointwise fixed by $\pi$, the same is true for
this subcomplex $\Cone(\Berg{M})$, so that  the map $\pi$ is actually a homotopy inverse to the inclusion map $\Cone(\Berg{M}) \hookrightarrow \Delta''$, showing that $\pi$ is a {\it deformation retraction}. 

In fact, if one removes the vertex $x_{f(\varnothing)}$ from both $\Delta', \Delta''$, one finds that
$\Delta'' \setminus \{x_{f(\varnothing)}\}$ is a subdivision of $\Delta' \setminus \{x_{f(\varnothing)}\}$,
and the  simplicial map $\pi$ also restricts
to a deformation retraction
\begin{equation}
\label{de-coned-retraction-map}
\Delta'' \setminus \{x_{f(\varnothing)}\} \overset{\pi}{\longrightarrow}
\Berg{f}.
\end{equation}
This is depicted for the closure $f$ from Example~\ref{closure-example} in Figure~\ref{exampleofpi}. The top row shows 
$\Delta' \setminus \{x_{f(\varnothing)}\}$.
The bottom row depicts the subdivision $\Delta'' \setminus \{x_{f(\varnothing)}\}$ along with the simplicial map  $\pi$ indicated by directed arrows along edges;
to its right is the subcomplex $\Berg{f}$ onto which it retracts.
\end{remark}

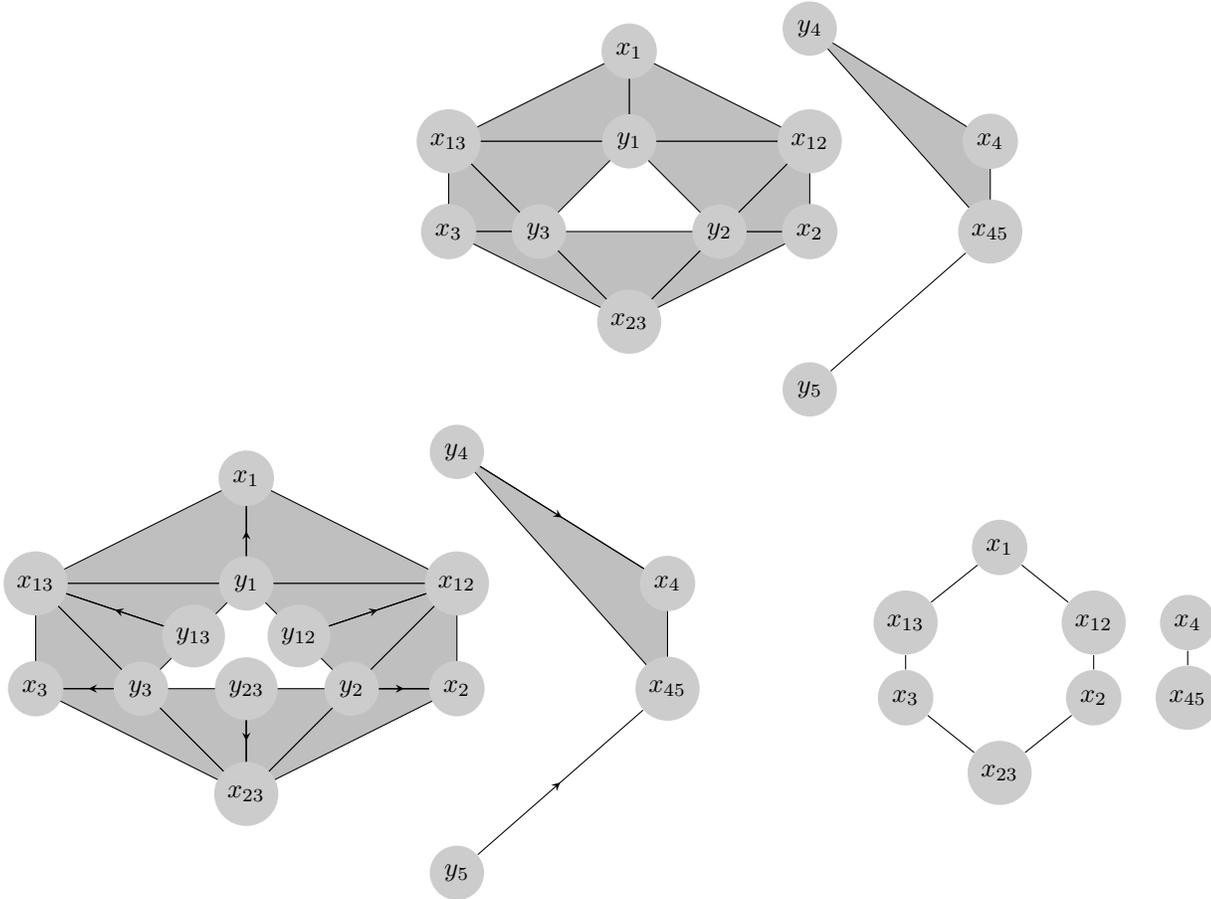
\begin{figure}[h] 

\begin{center}
\begin{minipage}{.2\textwidth}
 \hspace*{-0.6\linewidth}
\begin{tikzpicture}
  [scale=.30,auto=left,every node/.style={circle,fill=black!20}]
   \draw[fill=gray!50](24,8) -- (24,4) -- (16,13)-- (24,8);
  \draw[fill=gray!50](8,12) -- (8,8) -- (16,8)-- (8,12);
  \draw[fill=gray!50](8,12) -- (8,8) -- (0,8)-- (8,12);
  \draw[fill=gray!50](4,4) -- (8,8) -- (0,8)-- (4,4);
  \draw[fill=gray!50](4,4) -- (0,8) -- (0,4)-- (4,4);
  \draw[fill=gray!50](4,4) -- (8,0) -- (0,4)-- (4,4);
  \draw[fill=gray!50](4,4) -- (8,0) -- (12,4)-- (4,4);
  \draw[fill=gray!50](16,4) -- (8,0) -- (12,4)-- (16,4);
  \draw[fill=gray!50](16,4) -- (16,8) -- (12,4)-- (16,4);      
  \draw[fill=gray!50](8,8) -- (16,8) -- (12,4)-- (8,8);
  \node (y1) at (8,8) {$y_1$};
  \node (y2) at (12,4) {$y_2$};
  \node (y3) at (4,4) {$y_3$};
  \node (y4) at (16,13) {$y_4$};
  \node (y5) at (16,-3) {$y_5$};
  \node (x1) at (8,12) {$x_1$};
  \node (x2) at (16,4) {$x_2$};
  \node (x3) at (0,4) {$x_3$};
  \node (x12) at (16,8) {$x_{12}$};
  \node (x13) at (0,8) {$x_{13}$};
  \node (x23) at (8,0) {$x_{23}$};
  \node (x4) at (24,8) {$x_4$};
  \node (x45) at (24,4) {$x_{45}$};
  \foreach \from/\to in {y5/x45}  \draw (\from) -- (\to);
\end{tikzpicture}
\end{minipage}
\end{center}

\begin{center}

\begin{minipage}{0.1\textwidth}
 \hspace*{-3\linewidth}
\begin{tikzpicture}
  [scale=.35,auto=left,every node/.style={circle,fill=black!20}]
\tikzstyle arrowstyle=[scale=1]
\tikzstyle directed=[postaction={decorate,decoration={markings,    mark=at position .50 with {\arrow[arrowstyle]{stealth}}}}]
   \draw[fill=gray!50](24,8) -- (24,4) -- (16,13)-- (24,8);
  \draw[fill=gray!50](8,12) -- (8,8) -- (16,8)-- (8,12);
  \draw[fill=gray!50](8,12) -- (8,8) -- (0,8)-- (8,12);
  \draw[fill=gray!50](4,4) -- (0,8) -- (0,4)-- (4,4);
  \draw[fill=gray!50](4,4) -- (8,0) -- (0,4)-- (4,4);
  \draw[fill=gray!50](16,4) -- (8,0) -- (12,4)-- (16,4);
  \draw[fill=gray!50](16,4) -- (16,8) -- (12,4)-- (16,4);
 \draw[fill=gray!50](6,6) -- (8,8) -- (0,8)-- (6,6);
 \draw[fill=gray!50](6,6) -- (4,4) -- (0,8)-- (6,6);
  \draw[fill=gray!50](8,8) -- (16,8) -- (10,6)-- (8,8);
  \draw[fill=gray!50](12,4) -- (16,8) -- (10,6)-- (12,4);
  \draw[fill=gray!50](4,4) -- (8,0) -- (8,4)-- (4,4);
  \draw[fill=gray!50](12,4) -- (8,0) -- (8,4)-- (12,4);
   \node (y1) at (8,8) {$y_1$};
  \node (y2) at (12,4) {$y_2$};
  \node (y3) at (4,4) {$y_3$};
  \node (y4) at (16,13) {$y_4$};
  \node (y5) at (16,-3) {$y_5$};
  \node (x1) at (8,12) {$x_1$};
  \node (x2) at (16,4) {$x_2$};
  \node (x3) at (0,4) {$x_3$};
  \node (x12) at (16,8) {$x_{12}$};
  \node (x13) at (0,8) {$x_{13}$};
  \node (x23) at (8,0) {$x_{23}$};
  \node (x4) at (24,8) {$x_4$};
  \node (x45) at (24,4) {$x_{45}$};
  \node (y12) at (10,6) {$y_{12}$};
  \node (y13) at (6,6) {$y_{13}$};
  \node (y23) at (8,4) {$y_{23}$};
    \draw [directed] (y1) -- (x1);
    \draw [directed] (y2) -- (x2);
    \draw [directed] (y3) -- (x3);
    \draw [directed] (y12) -- (x12);
    \draw [directed] (y13) -- (x13);
    \draw [directed] (y23) -- (x23);
    \draw [directed] (y4) -- (x4);
    \draw [directed] (y5) -- (x45);
\end{tikzpicture}
\end{minipage}
\qquad\qquad\qquad\qquad\qquad\qquad\qquad 
\begin{minipage}{0.1\textwidth}
\begin{tikzpicture}
  [scale=.25,auto=left,every node/.style={circle,fill=black!20}]
\tikzstyle arrowstyle=[scale=1]
\tikzstyle directed=[postaction={decorate,decoration={markings,    mark=at position .50 with {\arrow[arrowstyle]{stealth}}}}]
  \node (x1) at (5,12) {$x_1$};
  \node (x2) at (10,4) {$x_2$};
  \node (x3) at (0,4) {$x_3$};
  \node (x12) at (10,8) {$x_{12}$};
  \node (x13) at (0,8) {$x_{13}$};
  \node (x23) at (5,0) {$x_{23}$};
  \node (x4) at (15,8) {$x_4$};
  \node (x45) at (15,4) {$x_{45}$};
 \foreach \from/\to in {x1/x12,x1/x13,x2/x12,x2/x23,x3/x13,x3/x23,x4/x45}
    \draw (\from) -- (\to);
\end{tikzpicture}
\end{minipage}

\end{center}

\label{exampleofpi}
\caption{
Continuing
Example~\ref{closure-example},
the deleted subcomplex
$\Delta' \setminus \{x_{f(\varnothing)}\}$
is shown at top.
The second line shows its subdivision $\Delta'' \setminus \{x_{f(\varnothing)}\}$,
along with directed arrows indicating the
retraction $\pi:\Delta'' \setminus \{x_{f(\varnothing)}\} \rightarrow \Berg{f}$ from \eqref{de-coned-retraction-map}.
}
\end{figure}


\section*{Acknowledgements}
This project was partially supported by NSF RTG grant DMS-1148634, as part of the University of Minnesota School of Mathematics Summer 2021 REU program. The authors thank Trevor Karn and Sasha Pevzner for presenting useful background and for valuable edits. 

\bibliography{bibliography}{}
\bibliographystyle{amsplain}

\end{document}